\documentclass[12pt]{article}
\usepackage{amsmath, amssymb, amsthm}
\usepackage[margin=1in]{geometry}
\usepackage{hyperref}
\usepackage{orcidlink}
\usepackage{authblk}
\usepackage{mathtools}

\newtheorem{theorem}{Theorem}[section]
\newtheorem{proposition}[theorem]{Proposition}
\newtheorem{lemma}[theorem]{Lemma}

\newtheorem{corollary}[theorem]{Corollary}
\newtheorem{remark}[theorem]{Remark}
\newtheorem{definition}[theorem]{Definition}
\newtheorem*{theorem*}{Theorem}

\title{Closed Formulas for $\eta$-Corrections in the Once Punctured Torus}
\author{Nelson Abdiel Col\'on Vargas \orcidlink{0009-0009-9038-7328}}

\affil{North Carolina State University \\ \texttt{nacolonv@ncsu.edu}}

\affil{University of Cambridge}

\date{}

\begin{document}

\maketitle

\begin{abstract}
We study $\eta$-correction terms in the Kauffman bracket skein algebra of the once-punctured torus $K_t(\Sigma_{1,1})$. While the Frohman--Gelca product-to-sum rule gives an explicit multiplication formula on the closed torus, the once-punctured torus introduces correction terms in the ideal $(\eta)$. We give a closed formula for the Chebyshev-threaded family generated by the primitive determinant-two pair
\[
P_n=T_n((1,2))\cdot(1,0).
\]
The correction $\epsilon_n$ has an explicit Chebyshev expansion whose coefficients factor as geometric sums in $t^{\pm4}$ and whose terms are governed by a parity pattern arising from the Chebyshev recurrence.

We also treat a primitive maximal-thread regime, in which one Frohman--Gelca summand is fully threaded and the other is simple or doubly covered. In this case the discrepancy is an explicit $\eta$-linear cascade with Chebyshev $S$-coefficients, lowering the thread degree by two at each step. These formulas recover the relevant low-determinant behavior and give compact closed multiplication rules for structured threaded families in $K_t(\Sigma_{1,1})$.
\end{abstract}

\section{Introduction}

The Kauffman bracket skein algebra $K_t(F)$ encodes topological information about framed links in the thickened surface $F\times[0,1]$ by imposing local skein relations and a stacking product. When $F$ is a torus, Frohman and Gelca \cite{frohmangelca} famously showed that multiplication of (threaded) simple closed curves admits a clean product-to-sum rule expressed in terms of the noncommutative torus. Passing to a \emph{punctured} torus, however, introduces a central element $\eta$ represented by a small loop about the puncture; resolving crossings that encounter the puncture generates additional terms lying in the ideal $(\eta)$, and the elegant closed-torus formula fractures into a main Frohman--Gelca part plus an $\eta$-correction. 
Understanding the structure of these correction terms is central to both the internal algebraic theory and to computational applications in topological quantum field theories, quantum character varieties, and quantum simulation architectures.

Recent work has established fundamental structural properties of $K_t(\Sigma_{1,1})$ through connections to other areas of mathematics. Bousseau \cite{bousseau} proved the Thurston positivity conjecture for $K_t(\Sigma_{1,1})$ by connecting the skein algebra to enumerative geometry, realizing the structure constants through counts of curves on complex cubic surfaces using quantum scattering diagrams. Similarly, Queffelec \cite{queffelec} established a general positivity result for the Jones-Wenzl basis on all orientable surfaces (other than the torus) by proving the full functoriality of $\mathfrak{gl}_{2}$ Khovanov homology for webs and foams. These geometric and categorical approaches establish the existence of important structural properties.

Despite substantial progress, explicit closed descriptions of $\eta$-corrections are available only in restricted regimes. For parallel curves ($|\det|=0$), the problem reduces to Chebyshev identities with no correction terms. Cho \cite{cho} established that products of simple curves with $|\det|=\pm2$ produce a single $\eta$. More recently, Wang--Wong \cite{wangwong} gave a fast recursive method for computing the Frohman--Gelca discrepancy on $\Sigma_{1,1}$. The aim of the present paper is to identify structured threaded families for which these correction terms admit closed symbolic formulas, exposing the combinatorial mechanism by which $\eta$-terms are created, propagated, and canceled.

Let $C=(p,q)$ denote a primitive simple closed curve on $\Sigma_{1,1}$ (identified with a homology class in $H_1(\Sigma_{1,1};\mathbb{Z})\cong\mathbb{Z}^2$ with $\gcd(p,q)=1$). Threaded curves arise by Chebyshev evaluation: if $k>1$ and $(kp,kq)$ is a multiple of $C$, then we define $(kp,kq)_T := T_k(C)$. When we multiply a threaded curve $T_n(C_2)$ by a simple curve $C_1$ whose algebraic intersection with $C_2$ has modest absolute value, even for $|\det(C_1,C_2)|=\pm2$, the naive approach of expanding $T_n$ via its recurrence quickly generates a tangle of intermediate products, each possibly introducing puncture loops. Manual calculation becomes uninformative just beyond the smallest values of $n$.

This paper makes progress on this challenge by providing closed formulas for significant families of threaded products. Our main contributions are threefold. First, we provide a closed formula for the canonical infinite family of products
\[
P_n := T_n((1,2)) \ast(1,0),
\]
whose underlying simple curves $(1,2)$ and $(1,0)$ have determinant $-2$. Equivalently, \(P_n=(n,2n)_T\cdot(1,0)\), so the threaded product itself has determinant \(-2n\); the determinant-two condition belongs to the primitive pair that generates the family, not to the threaded lattice labels.

Writing $L_k = \sum_{\ell=-k}^{k} t^{4\ell}$, we prove that for $n\ge 3$,
\begin{equation}\label{eq:intro-Pn-main}
P_n = t^{-2n}(n+1,2n)_T + t^{2n}(n-1,2n)_T + \epsilon_n,
\end{equation}
where the correction term is
\begin{equation}\label{eq:intro-Pn-eps}
\epsilon_n = \eta \sum_{k=0}^{\lfloor (n-1)/2\rfloor} \left( T_{n-1-2k}((1,2)) - \delta_{n-1-2k,0} \right) L_k.
\end{equation}

Second, Cho \cite{cho} showed that primitive pairs of curves with $|\det|=2$ are carried by mapping-class-group diffeomorphisms to the standard determinant-two pair. Consequently, the closed formula for the standard family $P_n$ gives a normal form for the determinant-two threaded case, transported by the induced skein-algebra automorphism. This provides a closed symbolic description of the correction pattern for products of the form $C_1\ast T_k(C_2)$ whenever the underlying primitive pair satisfies $|\det(C_1,C_2)|=2$.

Third, we identify the geometric source of this pattern through what we call the \emph{primitive maximal-thread mechanism}. For primitive curves $(p,q)$ and $(r,s)$ with determinant $n = ps-qr \geq 2$, suppose that one of the Frohman--Gelca summands is threaded with degree $n$ (i.e., either $n = \gcd(p + r, q + s)$ or $n = \gcd(p - r, q - s)$). We prove that the product takes the closed form
\[
(p,q) \ast (r,s) = t^{\varepsilon n} T_{d_1}(C_+) + t^{-\varepsilon n} T_{d_2}(C_-) + \epsilon,
\]
where $C_+ = (p + r, q + s)$, $C_- = (p - r, q - s)$, and $d_1 = \gcd(p + r, q + s)$, $d_2 = \gcd(p - r, q - s)$ with $\min\{d_1, d_2\} \in \{1, 2\}$ and $\max\{d_1, d_2\} = n$, the correction term has the explicit form
\[
\epsilon = \eta \ast \sum_{j = 0}^{\left\lfloor \frac{n-2}{2} \right\rfloor} t^{\varepsilon(n - 2 - 2j)} \ast \left( T_{n - 2 - 2j}(C_*/n) - \delta_{n - 2 - 2j, \, 0} \right) \ast S_j\left(t^2 + t^{-2}\right),
\]
where the signs and curves are determined by which summand carries the maximal threading. 

The work builds on Cho's low-determinant computations \cite{cho}, 
$\eta$-degree bounds from our previous work with Frohman \cite{abdiel2014}, and the recursive framework of Wang--Wong \cite{wangwong}. Its main contribution is to isolate two structured regimes in which the $\eta$-corrections admit closed symbolic formulas, making the correction patterns visible beyond their recursive computation.

The paper is organized as follows. Section~\ref{sec:prelim} reviews skein algebra conventions, threading, and determinant-based correction behavior. Section~\ref{sec:gen-mult} introduces the family $P_n$, records low-$n$ computations, and proves the parity lemmas that control $\eta$ creation. Section~\ref{sec:general-formula} establishes the closed-form expression for $P_n$ and develops equivalent coefficient factorizations. Section~\ref{sec:max-thread} formulates and proves the primitive maximal-thread correction structure.

\section{Preliminaries}
\label{sec:prelim}
\subsection{The Kauffman Bracket Skein Algebra $K_t(F)$}
Given an oriented surface $F$, we consider framed links within the cylindrical space $F \times [0,1]$, which are embeddings of disjoint unions of annuli. Taking formal $\mathbb{C}$-linear combinations of these links and imposing the Kauffman bracket skein relations yields:
\begin{align*}
    \langle \raisebox{-0.35em}{\includegraphics[width=.5cm]{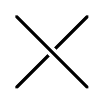}} \rangle &= t \langle \raisebox{-0.2em}{\includegraphics[width=.4cm]{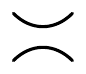}} \rangle + t^{-1} \langle \raisebox{-0.3em}{\includegraphics[width=.4cm]{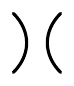}} \rangle \\
    \langle L \cup \bigcirc \rangle &= (-t^2 - t^{-2}) \langle L \rangle
\end{align*}
An algebra structure emerges from the product operation given by ``stacking'' along the $[0,1]$ direction.

\begin{definition}[Simple Closed Curve]
A \textbf{simple closed curve} on the torus is a curve that does not intersect itself. Algebraically, it is represented by a pair of coprime integers $(p,q)$, meaning their greatest common divisor is 1.
\[
\gcd(p,q) = 1
\]
\end{definition}

\begin{definition}[Composite (Threaded) Curve]
A \textbf{composite curve}, also called a threaded curve, is represented by a pair of integers $(p,q)$ that are not coprime, meaning their greatest common divisor is greater than 1, i.e., $\gcd(p,q) = k > 1$. It is not a single simple curve but is defined as the $k$-th Chebyshev polynomial, $T_k$, of the underlying simple curve $(p/k, q/k)$.
\[
(p,q)_T = T_k\left(\left(\frac{p}{k}, \frac{q}{k}\right)_T\right)
\]
\end{definition}

\subsection{The Product-to-Sum Formulas}
\begin{theorem}[Product-to-sum Formula in $K_t(\Sigma_{1,0})_T ${\cite[Theorem 4.1]{frohmangelca}}]
For any integers $p,q,r,s$, one has:
\[
(p,q)_T \ast (r,s)_T = t^{\left|\begin{smallmatrix} p & q \\ r & s \end{smallmatrix}\right|}(p+r,q+s)_T + t^{-\left|\begin{smallmatrix} p & q \\ r & s \end{smallmatrix}\right|}(p-r,q-s)_T.
\]
\end{theorem}
In previous work with Frohman \cite{abdiel2014}, we established the following product-to-sum formula and a related bound on the error term for $K_t(\Sigma_{1,1})$.

\begin{theorem}[Product-to-Sum Formula in $K_t(\Sigma_{1,1})$ {\cite[Theorem 6.1]{abdiel2014}}]
Let $(p,q)_T, (r,s)_T \in K_N(\Sigma_{1,1})$. Then
\[
(p,q)_T \ast (r,s)_T
= t^{\begin{vmatrix} p & q \\ r & s \end{vmatrix}} (p + r, q + s)_T
+ t^{- \begin{vmatrix} p & q \\ r & s \end{vmatrix}} (p - r, q - s)_T
+ \epsilon,
\]
where $\epsilon \in (\eta)$ and the weight of $\epsilon$ is less than or equal to $|p| + |q| + |r| + |s| - 4$.
\end{theorem}

\begin{lemma}[{\cite[Lemma 6.2]{abdiel2014}}] \label{lem:bound}
For $\epsilon$ as in the previous theorem, the highest power of $\eta$ appearing in $\epsilon$ is less than or equal to
\[
\left\lfloor \frac{\min\left\{ |p| + |r|,\, |q| + |s| \right\}}{2} \right\rfloor.
\]
\end{lemma}

\subsection{Correction Term Structure by Determinant Value}

The structure of the correction term $\epsilon$ depends critically on the determinant value:

\begin{itemize}
    \item \textbf{Determinant 0:} Products of parallel curves are resolved using Chebyshev polynomial identities, e.g., $T_n(C)T_m(C) = T_{n+m}(C) + T_{|n-m|}(C)$. Here, $\epsilon = 0$.
    
    \item \textbf{Determinant $\pm 1$: \cite{wangwong}} When $\det\begin{pmatrix}p & r \\ q & s\end{pmatrix} = \pm 1$, Wang--Wong show that the discrepancy term $\epsilon$ vanishes. In this case, the product reduces to the standard Frohman--Gelca two-term formula with no $\eta$-correction. Hence, $\epsilon = 0$.

    \item \textbf{Determinant $\pm$2: \cite{cho}} Resolved using results from Cho's thesis. The correction term is:
    \begin{equation}
    \epsilon = 
    \begin{cases}
    \eta & \text{if both } (p,q)_T \text{ and } (r,s)_T \text{ are simple curves} \\
    0 & \text{if at least one curve is not simple (composite)}
    \end{cases}
    \end{equation}
\end{itemize}
\textbf{The $\eta$-Bound Rule:} A direct consequence of the lemma~\ref{lem:bound} on the highest power of $\eta$ in $\epsilon$ provides a condition for when $\epsilon$ must be zero:
\begin{equation}
\text{If } \min(|p|+|r|, |q|+|s|) < 2 \text{, then } \epsilon = 0.
\end{equation}
\textbf{The Decomposition Method:} Products involving a composite (threaded) skein, such as $T_k(C)$, can be calculated by first expressing the skein as a polynomial of a simpler curve (e.g., $(2,0)_T = (1,0)_T^2 - 2$) and then applying the product rules to the sequence of simpler multiplications.

\section{General Multiplication Framework}
\label{sec:gen-mult}

\subsection{A Chebyshev-Threaded Family Generated by a Primitive Determinant-Two Pair}

We now focus on the family
\[
P_n=T_n((1,2))\cdot(1,0)=(n,2n)_T\cdot(1,0).
\]
Although the primitive pair \((1,2),(1,0)\) has determinant \(-2\), the threaded lattice labels \((n,2n)\) and \((1,0)\) have determinant \(-2n\). Thus the family is infinite in determinant while retaining a fixed primitive determinant-two seed. First, its determinant structure is simple enough that it can be solved algorithmically using the established rules. Second, it serves as a canonical example that can be generalized via diffeomorphism to a much larger class of products.

The skeins $(n, 2n)_T = T_n((1,2)_T)$ satisfy the Chebyshev recurrence relation. A careful derivation shows that this induces a recurrence for the products themselves:
\begin{equation*}
P_n = (1,2)_T \ast P_{n-1} - P_{n-2}
\end{equation*}
This recurrence relation is the key to finding a closed-form solution for the correction term $\epsilon_n$. We begin by establishing the base cases through direct computation.

\subsubsection*{Base Case Calculations}
The following results were derived using the decomposition method and the rules for small determinants. They form the foundation for our inductive proof.

\begin{itemize}
    \item \textbf{n=1:} $P_1 = t^{-2}(2,2)_T + t^{2}(0,2)_T + \eta$
    \item \textbf{n=2:} $P_2 = t^{-4}(3,4)_T + t^{4}(1,4)_T + (1,2)_T\eta$
    \item \textbf{n=3:} $P_3 = t^{-6}(4,6)_T + t^{6}(2,6)_T + (t^4 + t^{-4} + 1 + (2,4)_T)\eta$
    \item \textbf{n=4:} $P_4 = t^{-8}(5,8)_T + t^{8}(3,8)_T + [ (t^4 + t^{-4} + 1)(1,2)_T + (3,6)_T ]\eta$
\end{itemize}
See Appendix~\ref{appendix} for calculations.

\subsection{Parity and Determinant Properties for $P_n$}
\begin{lemma}[Constant Determinant Property]\label{lem:property}
For products of the form $(1,2)_T \ast (n \pm 1, 2n)_T$, the determinant is always $\pm 2$:
\begin{equation}
\det(((1,2)), (n+1, 2n)) = (1)(2n) - (2)(n+1) = -2
\end{equation}
\end{lemma}

\begin{lemma}[Parity-Dependent Simplicity]\label{lem:parity}
The generation of new $\eta$ terms in the recurrence depends on the parity of $n$:
\begin{itemize}
    \item If $n$ is even, then $\gcd(n+1, 2n) = 1$, so $(n+1, 2n)_T$ is simple and $(1,2)_T \ast (n+1, 2n)_T$ generates a new $\eta$ term.
    \item If $n$ is odd, then $\gcd(n+1, 2n) = 2$, so $(n+1, 2n)_T$ is composite and no new $\eta$ term is generated.
\end{itemize}
\end{lemma}

\section{The General Formula for $P_n$}
\label{sec:general-formula}

\begin{remark}[Chebyshev normalization]
Throughout, we use the Chebyshev normalization $T_0=2$ and $T_1(C)=C$. Thus expressions of the form
\[
T_m(C)-\delta_{m,0}
\]
are equal to $T_m(C)$ for $m>0$ and to the unit $1$ for $m=0$. This convention is responsible for the terminal scalar contribution in the $\eta$-correction.
\end{remark}

\begin{proposition}[General Formula for $P_n$]\label{thm:main}
For every $n\geq 1$, the product $P_n = (n, 2n)_T \ast (1,0)_T$ is given by:
\begin{equation}
P_n = t^{-2n}(n+1, 2n)_T + t^{2n}(n-1, 2n)_T + \epsilon_n
\end{equation}
where the correction term is:
\begin{equation}
\epsilon_n = \eta \left[ \sum_{k=0}^{\lfloor\frac{n-1}{2}\rfloor} \left( T_{n-1-2k}((1,2)) - \delta_{n-1-2k, 0} \right) 
\left( \sum_{l=-k}^{k} t^{4l} \right) 
\right]
\end{equation}
\end{proposition}

We first verify the cases $n=1$ and $n=2$ directly. The case $n=3$ is included to illustrate the first nontrivial cancellation pattern, and the induction then proves the formula for all $n\geq 1$.

\begin{proof}
The skeins $(n,2n)_T$ satisfy the Chebyshev recurrence relation $T_n(C) = C \ast T_{n-1}(C) - T_{n-2}(C)$ where $C = (1,2)_T$. This induces a recurrence for the products:
\begin{equation}
P_n = (1,2)_T \ast P_{n-1} - P_{n-2}
\end{equation}
\textbf{Cases for $k \in \{1,2\}$}
\begin{align}
\textbf{For } k=1: \quad P_1 &= (1,2)_T \ast (1,0)_T \\
&= t^{-2}(2,2)_T + t^2(0,2)_T + \eta
\end{align}

\begin{align}
\textbf{For } k=2: \quad P_2 &= (2,4)_T \ast (1,0)_T \\
& = ((1,2)_T^2 - 2) \ast(1,0)_T \\
& = (1,2)\ast[(1,2)_T\ast(1,0)_T] - 2(1,0)_T \\
& = (1,2)\ast[t^{-2}(2,2)_T + t^2(0,2)_T + \eta] - 2(1,0)_T \\
&= t^{-4}(3,4)_T + t^4(1,4)_T + (1,2)_T\eta
\end{align}
\textbf{Base Case}:
\noindent We verify the formula for $k=3$:

\begin{align}
\textbf{For } k=3: \quad P_3 &= (3,6)_T \ast (1,0)_T \\
&= t^{-6}(4,6)_T + t^6(2,6)_T + (t^4 + t^{-4} + 1 + (2,4)_T)\eta
\end{align}
The formula gives:
\begin{align}
\epsilon_3 &= \eta \left[T_2((1,2)) + (T_0((1,2))-1)(t^4 + 1 + t^{-4})\right] \\
&= \eta \left[(2,4)_T + (t^4 + 1 + t^{-4})\right]
\end{align}
which matches. See Appendix~\ref{appendix} for detailed calculations of $n \in \{3,4,5\}$.

\noindent\textbf{Inductive step.}
Assume the formula holds for all $k\le n$. We prove the formula for $P_{n+1}$ using
\[
P_{n+1}=(1,2)_T\cdot P_n-P_{n-1}.
\]


The main product-to-sum terms for $P_{n+1}$ are generated from the multiplication of $(1,2)_T$ with the main terms of $P_n$. As established by Lemma~\ref{lem:parity} and Lemma~\ref{lem:property}, the determinants of these products are always $\pm 2$. The main terms of $(1,2)_T \ast P_n$ correctly produce the main terms of $P_{n+1}$ plus the main terms of $P_{n-1}$. When we subtract $P_{n-1}$, these latter terms cancel perfectly, leaving the correct main terms for $P_{n+1}$. The core of the proof is to show that the correction terms also obey the recurrence. The recurrence for the epsilon term is:
\begin{equation*}
\epsilon_{n+1} = (\text{new } \eta \text{ terms}) + (1,2)_T \epsilon_n - \epsilon_{n-1}
\end{equation*}
The ``new $\eta$ terms'' are generated from the product $(1,2)_T \ast P_n$. Specifically, from the terms $t^{-2n}(1,2)_T(n+1,2n)_T$ and $t^{2n}(1,2)_T(n-1,2n)_T$. The determinants of these inner products are always $\pm 2$. A new $\eta$ is generated if and only if the curves are simple.
\begin{itemize}
    \item If $n$ is \textbf{even}, then $n\pm 1$ is odd. This means $\gcd(n\pm 1, 2n) = 1$, so the curves $(n\pm 1, 2n)_T$ are simple. Both products generate an $\eta$. The total contribution is $t^{-2n}(\eta) + t^{2n}(\eta) = (t^{2n}+t^{-2n})\eta$.
    \item If $n$ is \textbf{odd}, then $n\pm 1$ is even. This means $\gcd(n\pm 1, 2n) = 2$, so the curves $(n\pm 1, 2n)_T$ are composite. Neither product generates an $\eta$. The contribution is 0.
\end{itemize}
\paragraph{Case 1: $n$ is odd.}
In this case, no new $\eta$ terms are generated, so the recurrence is
\[
\epsilon_{n+1} = (1,2)_T \ast \epsilon_n - \epsilon_{n-1}.
\]
To prove this, we expand $(1,2)_T \ast \epsilon_n$ using the Kronecker delta formula for $\epsilon_n$ and the Chebyshev identity
\[
T_1 \ast T_m = T_{m+1} + T_{m-1}.
\]
Let
\[
L_k = \sum_{\ell = -k}^k (t^4)^\ell.
\]
Then:
\begin{align*}
(1,2)_T \ast \epsilon_n 
&= (1,2)_T \ast \eta \left[ \sum_{k=0}^{\left\lfloor\frac{n-1}{2}\right\rfloor} \left( T_{n-1 - 2k}(1,2)_T - \delta_{n-1 - 2k, 0} \right) L_k \right] \\
&= \eta \left[ \sum_{k=0}^{\left\lfloor\frac{n-1}{2}\right\rfloor} \left( (1,2)_T \ast T_{n-1 - 2k}(1,2)_T - (1,2)_T \ast \delta_{n-1 - 2k, 0} \right) L_k \right] \\
&= \eta \left[ \sum_{k=0}^{\left\lfloor\frac{n-1}{2}\right\rfloor} \left( T_{n - 2k}(1,2)_T + T_{n - 2 - 2k}(1,2)_T - (1,2)_T \ast \delta_{n-1 - 2k, 0} \right) L_k \right].
\end{align*}

This expression splits into two summations, which we denote by $\text{Sum A}$ and $\text{Sum B}$, plus a correction term from the Kronecker delta. Since $n$ is odd, the delta is non-zero only for the final term of the sum where $n - 1 - 2k = 0$, i.e., $k = \frac{n-1}{2}$.

\begin{align*}
\text{Sum A} &= \eta \left[ \sum_{k=0}^{\left\lfloor\frac{n-1}{2}\right\rfloor} T_{n - 2k}(1,2)_T \ast L_k \right], \\
\text{Sum B} &= \eta \left[ \sum_{k=0}^{\left\lfloor\frac{n-1}{2}\right\rfloor} T_{n - 2 - 2k}(1,2)_T \ast L_k \right].
\end{align*}

$\text{Sum A}$ corresponds exactly to the definition of $\epsilon_{n+1}$ (note that $n+1$ is even, so its delta term vanishes). $\text{Sum B}$ cancels with the Kronecker delta correction and the expression for $\epsilon_{n-1}$.

The result is the identity:
\[
(1,2)_T \ast \epsilon_n = \epsilon_{n+1} + \epsilon_{n-1},
\]
which rearranges to give the desired recurrence relation:
\[
\epsilon_{n+1} = (1,2)_T \ast \epsilon_n - \epsilon_{n-1}.
\]

\paragraph{Case 2: $n$ is even.}
In this case, a new term $(t^{2n} + t^{-2n})\eta$ is generated. Let $n=2m$. The recurrence for the correction term is:
\begin{equation*}
\epsilon_{n+1} = (t^{4m} + t^{-4m})\eta + (1,2)_T \epsilon_{2m} - \epsilon_{2m-1}
\end{equation*}
We expand the term $(1,2)_T\epsilon_{2m}$. Note that for $n=2m$, the index $n-1-2k$ is always odd, so the Kronecker delta in the formula for $\epsilon_{2m}$ is always zero.
\begin{align*}
(1,2)_T \epsilon_{2m} &= (1,2)_T \ast \eta \left[ \sum_{k=0}^{m-1} T_{2m-1-2k}(1,2) L_k \right] \\
&= \eta \left[ \sum_{k=0}^{m-1} (T_{2m-2k}(1,2)_T + T_{2m-2-2k}(1,2)) L_k \right] \\
&= \eta \left[ \sum_{k=0}^{m-1} T_{2m-2k}(1,2)L_k \right] + \eta \left[ \sum_{k=0}^{m-1} T_{2m-2-2k}(1,2)L_k \right]
\end{align*}
The second summation is almost $\epsilon_{2m-1}$. For $n=2m-1$, the index $n-2-2k = 2m-3-2k$ can be zero, so we must use the Kronecker delta form:
\begin{equation*}
\epsilon_{2m-1} = \eta \left[ \sum_{k=0}^{m-1} (T_{2m-2-2k}(1,2)_T - \delta_{2m-2-2k,0}) L_k \right]
\end{equation*}
Substituting this, we find that the expression $(1,2)_T\epsilon_{2m} - \epsilon_{2m-1}$ simplifies to:
\begin{equation*}
\eta \left[ \left( \sum_{k=0}^{m-1} T_{2m-2k}(1,2)L_k \right) + L_{m-1} \right]
\end{equation*}
Now, we add the new term back into the recurrence for $\epsilon_{n+1}$:
\begin{equation*}
\epsilon_{n+1} = (t^{4m} + t^{-4m})\eta + \eta \left[ \left( \sum_{k=0}^{m-1} T_{2m-2k}(1,2)L_k \right) + L_{m-1} \right]
\end{equation*}
We use the identity $L_m = L_{m-1} + t^{4m} + t^{-4m}$ to combine the terms:
\begin{align*}
\epsilon_{n+1} &= \eta \left[ \left( \sum_{k=0}^{m-1} T_{2m-2k}(1,2)L_k \right) + (t^{4m} + t^{-4m} + L_{m-1}) \right] \\
&= \eta \left[ \left( \sum_{k=0}^{m-1} T_{2m-2k}(1,2)L_k \right) + L_m \right]
\end{align*}
This is almost the target formula for $\epsilon_{n+1}$. The target is $\eta \left[ \sum_{k=0}^{m} (T_{2m-2k} - \delta_{2m-2k,0}) L_k \right]$. Expanding this gives:
\begin{equation*}
\eta \left[ \left( \sum_{k=0}^{m-1} T_{2m-2k}(1,2)L_k \right) + (T_0-1)L_m \right] = \eta \left[ \left( \sum_{k=0}^{m-1} T_{2m-2k}(1,2)L_k \right) + L_m \right]
\end{equation*}
The expressions match. This completes the proof for the case where $n$ is even.
\end{proof}

\begin{remark}
Notice that 
\[
L_k = \left( \sum_{l=-k}^{k} t^{4l} \right) = \left( \sum_{l=0}^{k} S_{2l}(t^2+t^{-2}) \right).
\]
In the following section we will use the $S_j$ notation for coefficients.
\end{remark}

\begin{corollary}[Determinant-two transport]\label{cor:det-two-transport}
Let $C_1$ and $C_2$ be primitive simple curves on $\Sigma_{1,1}$ with
\[
|\det(C_1,C_2)|=2.
\]
Then the closed formula for the standard family
\[
P_n=T_n((1,2))\ast(1,0)
\]
transports, under the mapping-class-group action on $K_t(\Sigma_{1,1})$, to a closed formula for the corresponding product
\[
T_n(C_2)\ast C_1,
\]
up to the order determined by the chosen diffeomorphism.
\end{corollary}

\begin{proof}
By Cho's classification of determinant-two primitive pairs, there is a mapping-class-group element carrying the pair $(C_2,C_1)$ to the standard pair $((1,2),(1,0))$, possibly after changing orientations. The mapping-class group acts on $K_t(\Sigma_{1,1})$ by algebra automorphisms and preserves the peripheral element $\eta$. Applying the inverse automorphism to the standard formula for $P_n$ therefore gives the corresponding formula for $T_n(C_2)\ast C_1$. Since the automorphism preserves products, threading, and $\eta$-corrections, the transported expression has the same coefficient structure as \eqref{eq:intro-Pn-main}--\eqref{eq:intro-Pn-eps}.
\end{proof}

\begin{remark}[Order of the factors]
If one instead wants the product
\[
C_1\ast T_n(C_2),
\]
one applies the same transport argument after choosing the diffeomorphism that identifies the ordered pair with the flipped standard pair. Equivalently, the reversed-order formula is obtained from the same determinant-two calculation with the determinant sign reversed; the Frohman--Gelca coefficients are interchanged, while the parity pattern governing the $\eta$-correction is unchanged.
\end{remark}

\section{The Primitive Maximal-Thread Case}
\label{sec:max-thread}

We work with the Kauffman parameter $t$ (so $\bigcirc=-t^2-t^{-2}$). On $\Sigma_{1,1}$ the
peripheral loop $\eta$ is \emph{central} in $K_t(\Sigma_{1,1})$, and for any fixed primitive simple
closed curve $C$ the family $\{T_m(C)\}_{m\ge0}$ is \emph{linearly independent} over
$\mathbb{C}[t^{\pm1},\eta]$ (e.g.\ via the Frohman--Gelca embedding or the quantum trace \cite{frohmangelca,bonahonwong}). We use these facts tacitly.

We analyze the case where one Frohman--Gelca summand is maximally threaded. Let
\[
\alpha=(p,q),\qquad \beta=(r,s),\qquad n=\det\!\begin{pmatrix}p&r\\ q&s\end{pmatrix}\ge2,
\]
and set $C_\pm=\alpha\pm\beta$, $d_\pm=\gcd(C_\pm)$. Assume
\[
\max\{d_+,d_-\}=n\qquad\text{and}\qquad \min\{d_+,d_-\}\in\{1,2\}.
\]
Write $\varepsilon=\operatorname{sgn}(n)$ and let $C_\ast$ denote the \emph{maximally threaded}
Frohman--Gelca summand, i.e.\ $d_\ast=n$.

\begin{proposition}[Threaded $\eta$-Correction Structure]\label{prop:threaded}
Let $\alpha=(p,q)$ and $\beta=(r,s)$ be primitive on $\Sigma_{1,1}$ with
$n=\det\!\begin{psmallmatrix}p&r\\ q&s\end{psmallmatrix}\ge2$.
Set $C_{\pm}=\alpha\pm\beta$ and $d_{\pm}=\gcd(C_{\pm})$.
Assume $\max\{d_+,d_-\}=n$ and $\min\{d_+,d_-\}\in\{1,2\}$.
Then, in $K_t(\Sigma_{1,1})$,
\[
\alpha\ast\beta
= t^{\varepsilon n}\,T_{d_+}(C_+/d_+)\;+\;t^{-\varepsilon n}\,T_{d_-}(C_-/d_-)\;+\;\epsilon,
\]
where $\varepsilon=\operatorname{sgn}(n)$ and the correction term is
\[
\epsilon
= \eta\ \ast\sum_{j=0}^{\lfloor (n-2)/2\rfloor}
t^{\varepsilon(n-2-2j)}\Big(T_{n-2-2j}(C_\ast/n)-\delta_{n-2-2j,0}\Big)\,S_j(t^2+t^{-2}),
\]
with $C_\ast$ the maximally threaded Frohman-Gelca summand ($C_\ast=n\mu_\ast$ with $\mu_\ast$ primitive).
\end{proposition}

\paragraph{Preliminaries for the proof.}
Before proving Proposition~\ref{prop:threaded}, we record two auxiliary ingredients that will be
used repeatedly. First, an arithmetic normal form (and its corollary) places any pair
$(\alpha,\beta)$ in a canonical $SL_2(\mathbb Z)$-frame where one Frohman-Gelca summand is $n$
times a primitive and makes the thread degrees $d_\pm$ explicit. Second, a Chebyshev
``cascade'' identity solves the specialized three-term recurrence that arises from the
Recursion Rule once we fix the maximally threaded direction.


\begin{lemma}[SL$_2(\mathbb Z)$ normal form and thread degrees]\label{lem:NF}
Let $u=(p,q)$ and $v=(r,s)$ be primitive with $\det(u,v)=n\ge2$. There is $M\in\mathrm{SL}_2(\mathbb Z)$
such that $M u=(1,0)$ and $M v=(a,n)$ with $0\le a<n$ and $\gcd(a,n)=1$. In this normal form
\[
d_+=\gcd(1+a,n),\qquad d_-=\gcd(1-a,n),
\]
where $d_\pm=\gcd(u\pm v)$ are the thread degrees of the Frohman-Gelca summands.
\end{lemma}

\begin{proof}
This is Smith normal form for the $2\times2$ integer matrix $[u\ v]$, together with the fact that
$\gcd$ of coordinates is preserved under unimodular maps. The formulas for $d_\pm$ are immediate.
\end{proof}

\begin{corollary}[Characterization of the small co-thread regime]\label{cor:small}
With notation as in Lemma~\ref{lem:NF}, the following are equivalent:
\begin{enumerate}
\item $\max\{d_+,d_-\}=n$ and $\min\{d_+,d_-\}\in\{1,2\}$.
\item $a\equiv \pm 1 \pmod n$.
\item $u\equiv \pm v \pmod n$ (componentwise congruence).
\end{enumerate}
In this case the non‑maximal thread degree is $\gcd(n,2)$; in particular it equals $1$ if $n$ is odd and $2$ if $n$ is even.
\end{corollary}

\begin{proof}
If $a\equiv1\pmod n$, then $d_-=\gcd(1-a,n)=n$ and $d_+=\gcd(1+a,n)=\gcd(2,n)\in\{1,2\}$. If
$a\equiv-1\pmod n$, the roles swap. Conversely, if $\max\{d_+,d_-\}=n$, then either $1+a\equiv0$ or
$1-a\equiv0\pmod n$, i.e. $a\equiv\pm1\pmod n$. For (ii)\(\Rightarrow\)(iii), if $Mu=(1,0)$ and $Mv=(a,n)$, then $(1,0)\equiv\pm(a,n)\pmod n$ is equivalent to $u\equiv\pm v\pmod n$ by unimodularity of $M$. The reverse implication is identical.
\end{proof}

\begin{remark}[Quick test and parametrization]
Practically: given primitive $u,v$ with $n = \det(u,v) \geq 2$, compute $u\pm v$ modulo $n$. If one of
$u\pm v$ vanishes mod $n$, we are in the regime of Proposition~\ref{prop:threaded}
with the other thread degree $=\gcd(n,2)$. Equivalently, every such pair is SL$_2(\mathbb Z)$-equivalent to
$\big((1,0),(\pm1,n)\big)$; a convenient parametrization is $v=\pm u + n w$, where $w$ is any primitive vector with $\det(u,w)=1$.
\end{remark}

\begin{lemma}[Chebyshev cascade identity]\label{lem:cascade}
Let $\mu=(0,1)$ and put $\varepsilon=\operatorname{sgn}(n)$, $x:=t^2+t^{-2}$.
Define
\[
G_n \;:=\;\sum_{j=0}^{\lfloor (n-2)/2\rfloor}
t^{\varepsilon(n-2-2j)}\Big(T_{n-2-2j}(\mu)-\delta_{n-2-2j,0}\Big)\,S_j(x).
\]
Then
\begin{equation}\label{eq:cascade}
\mu_T\ast G_n \;=\; t^{\varepsilon}\,G_{n-1} \;+\; t^{-\varepsilon}\,G_{n+1} \;-\; t^{-\varepsilon n}\,.
\end{equation}
\end{lemma}

\begin{proof}
Write $k=n-2-2j$ (so $k\ge0$ and $k\equiv n\bmod 2$). Using
\[
\mu_T T_k(\mu)=T_{k+1}(\mu)+T_{k-1}(\mu)\quad(k\ge1),\qquad
S_{j+1}(x)=xS_j(x)-S_{j-1}(x),\quad S_0=1,\ S_1=x,
\]
we have, termwise,
\[
\mu_T\!\left(t^{\varepsilon k}\,[T_k(\mu)-\delta_{k,0}]\,S_j(x)\right)
= t^{\varepsilon k}\big(T_{k+1}+T_{k-1}-\delta_{k,0}\big)S_j(x).
\]
Reindex the $j$-sum: the $T_{k+1}$ part is the $j$th summand of $t^{\varepsilon}G_{n-1}$; the $T_{k-1}$ part
is the $(j+1)$st summand of $t^{-\varepsilon}G_{n+1}$. The only residue is when $k=0$,
where $T_0(\mu)=2$ and we have subtracted $1$; this yields the unit $1=T'_0$ and produces the final
$-\,t^{-\varepsilon n}$ in \eqref{eq:cascade}. Summing over $j$ gives the claim.
\end{proof}

\begin{proof}[Proof of Proposition~\ref{prop:threaded}]
We use the Frohman-Gelca splitting in the Chebyshev-$T$ basis:
\begin{equation}\label{eq:FG}
(\alpha)_T(\beta)_T
= t^{|n|}(C_+)_T+t^{-|n|}(C_-)_T \;+\; \eta\,D(\alpha,\beta),
\end{equation}
where $D$ is the discrepancy (see \cite[Def.\ 2.5/Eq.\ (1)]{wangwong}).

\paragraph{Normalization.}
By Corollary~\ref{cor:small}, one of $C_\pm$ equals $n$ times a primitive direction.
Replacing $\beta$ by $-\beta$ if needed, assume $C_\ast=C_+=n\mu_\ast$.
By the mapping-class action (e.g. \cite[Lem.\ 3.2]{wangwong}), there exists $\phi\in SL_2(\mathbb Z)$ with
$\phi_*(\mu_\ast)=(0,1)=:\mu$; the action preserves \eqref{eq:FG} and the discrepancy.
Hence it suffices to prove the claim for $(\tilde\alpha,\tilde\beta):=\phi_*(\alpha,\beta)$ with
$\tilde C_+=n\mu$ and $d_-:=\gcd(\tilde C_-)\in\{1,2\}$, then undo $\phi_*$.

\paragraph{Specialized recurrence.}
Let $F_n:=D(\tilde\alpha,\tilde\beta)$ and $B:=D(\mu;\tilde C_-)$.
Apply the Recursion Rule (\cite[Thm.\ 3.1]{wangwong}) with $u=\mu$.
Since $D(\mu;\tilde C_+)=0$ for $\det=0$ (\cite[Lem.\ 2.6]{wangwong}), we obtain
\begin{equation}\label{eq:RR}
\mu_T\ast F_n \;=\; t^{\varepsilon}F_{n-1} \;+\; t^{-\varepsilon}F_{n+1} \;-\; t^{-\varepsilon n}\,B.
\end{equation}

\paragraph{Inhomogeneous term.}
Because $d_-\in\{1,2\}$, the small closed forms (\cite[Prop.\ 4.1-4.3]{wangwong}) give $B=T'_0=1$ in the $T'$-normalization.
With our $T_0=2$, we will compensate by the Kronecker subtraction in $G_n$ below.

\paragraph{Identification of the discrepancy.}
Define $G_n$ as in Lemma~\ref{lem:cascade}. By that lemma, $G_n$ satisfies \eqref{eq:RR} with $B=1$.
For $|n|\le1$ the discrepancy vanishes (\cite[Lem.\ 2.6-2.8]{wangwong}); for $n=2$ our subtraction at $T_0$ makes $G_2=0$.
Thus $F_n$ and $G_n$ agree on bases, and by the algorithm's uniqueness/termination
(\cite[Prop.\ 5.1, Rem.\ 5.2, Prop.\ 5.4]{wangwong}) we conclude $F_n\equiv G_n$.

Insert $D=G_n$ into \eqref{eq:FG} for $(\tilde\alpha,\tilde\beta)$ to obtain
\[
(\tilde\alpha)_T(\tilde\beta)_T
= t^{\varepsilon n}\,T_{n}(\mu) \;+\; t^{-\varepsilon n}\,T_{d_-}(\tilde C_-/d_-)
\;+\; \eta\,G_n.
\]
Undoing $\phi_*$ yields the displayed formula with $C_\ast/n=\mu_\ast$, i.e. with

$\epsilon= \eta G_n$ expressed in the direction of the maximally threaded summand.
\end{proof}

\begin{remark}[Coefficient dictionary]\label{rem:dictionary}
For $x=t^2+t^{-2}$,
\[
S_j(x)=\frac{t^{2(j+1)}-t^{-2(j+1)}}{t^2-t^{-2}},
\]
hence each summand of $\epsilon(\alpha,\beta)$ carries coefficient
$t^{\varepsilon(n-2-2j)}S_j(x)$, matching the rational coefficients that appear in the
fast algorithm simplifications in \cite[\S~4]{wangwong}.
\end{remark}

See Section~\ref{app:ww-example} for full computations and the comparison with the Wang-Wong algorithm.

\appendix
\label{appendix}

\section{Detailed Calculations}

\subsection*{Calculation for n=3}
The product is $P_3 = (3,6)_T \ast (1,0)_T$. We use the recurrence $P_3 = (1,2)_T \ast P_2 - P_1$.
\begin{align}
P_3 &= (1,2)_T \ast [t^{-4}(3,4)_T + t^{4}(1,4)_T + (1,2)_T\eta] - P_1
\end{align}
The intermediate products $|det|=2$ between simple curves, generating new $\eta$ terms.
\begin{align}
P_3 &= t^{-4}[t^{-2}(4,6)_T+t^{2}(2,2)_T+\eta] + t^{4}[t^{2}(2,6)_T+t^{-2}(0,2)_T+\eta] \\
&\quad + (1,2)_T^2\eta - P_1 \\[0.5em]
&= t^{-6}(4,6)_T+t^{-2}(2,2)_T+t^{-4}\eta + t^{6}(2,6)_T+t^{2}(0,2)_T+t^{4}\eta \\
&\quad + ((2,4)_T+2)\eta - P_1
\end{align}
We substitute $P_1$ back in. The main terms $t^{-2}(2,2)_T+t^{2}(0,2)_T$ cancel with the main terms of $-P_1$, and one $\eta$ cancels.
\begin{align}
P_3 = (3,6)_T \ast (1,0)_T &= t^{-6}(4,6)_T + t^{6}(2,6)_T \\
&\quad + [ t^4 + t^{-4} + 1 + (2,4)_T ]\eta
\end{align}

\subsection*{Calculation for n=4}

The product is $P_4 = (4,8)_T \ast (1,0)_T$. We use the recurrence $P_4 = (1,2)_T \ast P_3 - P_2$.

\begin{align}
P_4 &= (1,2)_T \ast [t^{-6}(4,6)_T + t^{6}(2,6)_T + (t^4 + t^{-4} + 1 + (2,4))\eta] - P_2
\end{align}

The intermediate products involve composite curves, so no new $\eta$ terms are generated.

\begin{align}
P_4 &= t^{-6}[t^{-2}(5,8)_T+t^{2}(3,4)] + t^{6}[t^{2}(3,8)_T+t^{-2}(1,4)_T] \\
&\quad + (1,2)_T(t^4+t^{-4}+1)\eta + (1,2)_T \ast (2,4)_T\eta - P_2 \\[0.5em]
&= t^{-8}(5,8)_T+t^{-4}(3,4)_T + t^{8}(3,8)+t^{4}(1,4)_T \\
&\quad + (t^4+t^{-4}+1)(1,2)_T\eta + ((3,6)_T+(1,2)_T)\eta - P_2
\end{align}

We substitute $P_2$ back in. The main terms $t^{-4}(3,4)_T$ and $t^{4}(1,4)_T$ cancel, and one $(1,2)\eta$ cancels.

\begin{align}
P_4 = (4,8)_T \ast (1,0)_T &= t^{-8}(5,8)_T + t^{8}(3,8)_T \\
&\quad + [ (t^4 + t^{-4} + 1)(1,2)_T + (3,6)_T ]\eta
\end{align}

\subsection*{Calculation for n=5}
The product is $P_5 = (5,10)_T \ast (1,0)_T$. We use the recurrence $P_5 = (1,2)_T \ast P_4 - P_3$. We start by expanding the term $(1,2)_T \ast P_4$:
\begin{align*}
(1,2)_T \ast P_4 = (1,2)_T \ast \left[ t^{-8}(5,8)_T + t^{8}(3,8)_T + \epsilon_4 \right]
\end{align*}
where $\epsilon_4 = [ (t^4 + t^{-4} + 1)(1,2)_T + (3,6)_T ]\eta$.

The intermediate products $(1,2)_T \ast (5,8)_T$ and $(1,2)_T \ast (3,8)_T$ are between simple curves with $|det|=\pm 2$, and thus they generate new $\eta$ terms:
\begin{align*}
(1,2)_T \ast P_4 &= t^{-8}\left[t^{-2}(6,10)_T+t^{2}(4,6)_T+\eta\right] + t^{8}\left[t^{2}(4,10)_T+t^{-2}(2,6)_T+\eta\right] \\
&\quad + (1,2)_T \ast \epsilon_4 \\[0.5em]
&= t^{-10}(6,10)_T+t^{-6}(4,6)_T+t^{-8}\eta + t^{10}(4,10)_T+t^{6}(2,6)_T+t^{8}\eta \\
&\quad + (1,2)_T \ast \left[ (t^4 + t^{-4} + 1)(1,2)_T + (3,6)_T \right]\eta
\end{align*}
We expand the coefficient of $\eta$:
\begin{align*}
(1,2)_T \ast \epsilon_4 &= \left[ (t^4 + t^{-4} + 1)(1,2)_T^2 + (1,2)_T \ast (3,6)_T \right]\eta \\
&= \left[ (t^4 + t^{-4} + 1)((2,4)_T+2) + ((4,8)_T+(2,4)_T) \right]\eta \\
&= \left[ (t^4 + t^{-4} + 2)(2,4)_T + 2(t^4 + t^{-4} + 1) + (4,8)_T \right]\eta
\end{align*}
Now we assemble the full expression for $P_5 = (1,2)_T \ast P_4 - P_3$ and substitute the known formulas for $P_4$ and $P_3$. The main terms $t^{-6}(4,6)_T$ and $t^{6}(2,6)_T$ from the expansion cancel perfectly with the main terms from $-P_3$. The remaining terms form the new correction term $\epsilon_5$:
\begin{align*}
\epsilon_5 &= (t^8+t^{-8})\eta + (1,2)_T\epsilon_4 - \epsilon_3 \\
&= (t^8+t^{-8})\eta + \left[ (t^4 + t^{-4} + 2)(2,4)_T + 2(t^4 + t^{-4} + 1) + (4,8)_T \right]\eta \\
&\quad - \left[ t^4 + t^{-4} + 1 + (2,4)_T \right]\eta
\end{align*}
Collecting the coefficients of the skeins inside the $\eta$ term:
\begin{itemize}
    \item Scalar part: $(t^8+t^{-8}) + 2(t^4+t^{-4}+1) - (t^4+t^{-4}+1) = t^8+t^{-8}+t^4+t^{-4}+1$
    \item $(2,4)_T$ part: $(t^4+t^{-4}+2)(2,4)_T - (2,4)_T = (t^4+t^{-4}+1)(2,4)_T$
    \item $(4,8)_T$ part: $(4,8)_T$
\end{itemize}
This gives the final expression for the product:
\begin{align*}
P_5 = (5,10)_T \ast (1,0)_T &= t^{-10}(6,10)_T + t^{10}(4,10)_T + \epsilon_5
\end{align*}
where
\begin{align*}
\epsilon_5 &= \eta \left[ (t^8+t^{-8}+t^4+t^{-4}+1) + (t^4+t^{-4}+1)(2,4)_T + (4,8)_T \right]
\end{align*}


\section{Normalizations and comparison with Wang--Wong}\label{app:norm-ww}
\begin{remark}[Normalization: \(T_0\) vs.\ \(T'_0\)]
After developing our formulas, we compared them with the Wang-Wong fast algorithm \cite[p.9]{wangwong} and noticed a notational normalization that is worth recording. We package the terminal peel contribution using
\[
(T_0 - \delta_{0,0})\qquad\text{with }T_0=2,
\]
so that the last term in the cascade contributes \(S_j(t^2+t^{-2})\ast 1\).
Wang-Wong adopt the commonly used Chebyshev variation \(T'_0=1\) and \(T'_k=T_k\) for \(k\ge 1\). Equivalently,
\[
S_j(t^2+t^{-2})\,(T_0-\delta_{0,0}) \;=\; S_j(t^2+t^{-2})\,T'_0
\]
and their printed \(T(0,0)\) should be read as the unit \(1\).
This dictionary eliminates the apparent factor‑of‑two ambiguity in the terminal \(\eta\)-term and makes our peel‑cascade coefficients coincide term‑by‑term with the Wang-Wong output.
\end{remark}

\begin{remark}[Agreement with Wang-Wong \cite{wangwong}]
Write $x=t^{2}+t^{-2}$. The Chebyshev identity
\[
S_j(x)=\frac{t^{2(j+1)}-t^{-2(j+1)}}{t^{2}-t^{-2}}
\]
shows that each peel-cascade coefficient in Proposition~\ref{prop:threaded}
\[
t^{\varepsilon\,(n-2-2j)}\,S_j(x)
\]
coincides with the coefficient produced by the Wang-Wong fast algorithm for the same term, namely
\[
t^{\varepsilon\,(n-2-2j)}\;\frac{t^{2(j+1)}-t^{-2(j+1)}}{t^{2}-t^{-2}}.
\]
Hence the maximal‑thread expansion matches the Wang-Wong output term‑by‑term. See Appendix~\ref{app:ww-example} for a worked example when $(\alpha,\beta)=(2,1),(3,4)$ and several others.
\end{remark}

\section{Maximal Thread Results and Comparison}\label{app:ww-example}

\begin{remark}[Relation to the Wang--Wong recursion]
Proposition~\ref{prop:threaded} may be viewed as a closed-form specialization of the Wang--Wong recursion in the primitive maximal-thread regime. In this regime the recursion has a single active threaded direction, and Lemma~\ref{lem:cascade} solves the resulting inhomogeneous recurrence explicitly. The coefficient dictionary in Remark~\ref{rem:dictionary} identifies the resulting terms with the fast-algorithm coefficients.
\end{remark}

\subsection*{Example: \texorpdfstring{$(4,3)\ast(0,1)$}{(4,3)\ast(0,1)} in the maximal-thread regime}

Let $\alpha=(4,3)$ and $\beta=(0,1)$. Then
\begin{align}
n &= \det\!\begin{pmatrix}4&0\\ 3&1\end{pmatrix} = 4, \\
C_{+} &= \alpha+\beta = (4,4), \quad d_{+} = \gcd(4,4) = 4 = n, \quad \mu_{+} = (1,1), \\
C_{-} &= \alpha-\beta = (4,2), \quad d_{-} = \gcd(4,2) = 2.
\end{align}

Thus the ``$+$'' Frohman--Gelca summand is maximally threaded ($d_{+}=n$), while the ``$-$'' summand has thread degree $2$. The closed-torus part is
\[
t^{4}\,T_{4}(1,1) + t^{-4}\,T_{2}(2,1) = t^{4}\,T(4,4) + t^{-4}\,T(4,2).
\]

By Proposition~\ref{prop:threaded}, the once-punctured torus product adds a peel-cascade from the threaded summand. With $\lfloor(n-2)/2\rfloor=1$ we have two levels $j=0,1$ and $\varepsilon=+1$:
\[
\eta\!\sum_{j=0}^{1} t^{\,4-2-2j}
\Big(T_{\,4-2-2j}(1,1)-\delta_{\,4-2-2j,0}\Big)\,S_j\!\big(t^{2}+t^{-2}\big).
\]

Using $S_{0}=1$ and $S_{1}(x)=x$ with $x=t^{2}+t^{-2}$, this equals
\[
\eta\left[ t^{2}\,T_{2}(1,1) + (T_{0}-1)\,S_1(x)\right].
\]

Altogether,
\begin{align}
(4,3)\ast(0,1) &= t^{-4}\,T(4,2) + t^{4}\,T(4,4) \nonumber \\
&\quad + \eta\left[ t^{2}\,T_{2}(1,1) + (T_{0}-1)\,S_1(x)\right], \quad x=t^{2}+t^{-2}.
\end{align}

Since $T_{0}=2$, the factor $(T_{0}-1)$ equals the unit, so the terminal term is simply $S_1(x)$.

\subsubsection*{Wang--Wong's Fast Algorithm output and simplification}

The Wang--Wong fast algorithm for the same input returns
\begin{align}
&t^{-4}\,T(4,2) + t^{4}\,T(4,4) \nonumber \\
&\quad + \eta\left[\begin{aligned}
&\frac{-t^{4}+t^{-4}}{-t^{2}+t^{-2}}\;T'(0,0) + t^{2}\,T(2,2)
\end{aligned}\right].
\end{align}

Using the paper's $T'$-normalization, $T'(0,0)=1$ (the unit), and
\[
\frac{-t^{4}+t^{-4}}{-t^{2}+t^{-2}} = \frac{t^{4}-t^{-4}}{t^{2}-t^{-2}} = S_1(x), \qquad T(2,2)=T_{2}(1,1),
\]
this simplifies to
\[
t^{-4}\,T(4,2)+t^{4}\,T(4,4) + \eta\left[ t^{2}\,T_{2}(1,1)+S_1(x)\right],
\]
which matches the maximal-thread peel-cascade expression above.

\subsubsection*{Reconciliation with Cho's \texorpdfstring{$|\det|=4$}{|det|=4} calculation}

Cho writes
\[
(4,3)\ast(0,1) = t^{-4}\,T(4,2)+t^{4}\,T(4,4) + \eta\big[t^{2}\,(2,2)_{T}+t^{-2}\,(0,0)_{S}\big].
\]

Changing basis and using $S_1(x)=\dfrac{t^{4}-t^{-4}}{t^{2}-t^{-2}}=t^{2}+t^{-2}$ and $T_{2}(1,1)=(1,1)^{2}-2$, we obtain
\[
\eta\left[t^{2}\,T_{2}(1,1) + S_1(x) \ast 1\right],
\]
which is exactly the same $\eta$-term as in the maximal-thread and Wang--Wong expressions (recall the unit convention $T'(0,0)=1$). Hence all three presentations agree term-by-term.

\subsection*{Example: \texorpdfstring{$(2,1)\ast(3,4)$}{(2,1)\ast(3,4)} in the maximal‑thread regime}
Let $\alpha=(2,1)$ and $\beta=(3,4)$. Then
\[
n=\det\!\begin{pmatrix}2&3\\ 1&4\end{pmatrix}=5,\qquad
C_{+}=\alpha+\beta=(5,5),\ d_{+}=\gcd(5,5)=5,\ \mu_{+}=(1,1),
\]
\[
C_{-}=\alpha-\beta=(-1,-3),\ d_{-}=\gcd(1,3)=1.
\]
Thus the “$+$” Frohman-Gelca summand is maximally threaded ($d_{+}=n$), while the “$-$” summand is simple. The closed‑torus part is
\[
t^{5}\,T_{5}(1,1)\;+\;t^{-5}\,T_{1}(-1,-3)
\;=\;t^{5}\,T_{5}(1,1)\;+\;t^{-5}\,(1,3),
\]
since $T_{1}$ is the identity and $(-1,-3)$ represents the same simple curve as $(1,3)$.

By Proposition~\ref{prop:threaded}, the once‑punctured torus product adds a peel‑cascade from the threaded summand. With $\lfloor(n-2)/2\rfloor=1$ we have two levels $j=0,1$:
\[
\eta\!\sum_{j=0}^{1} t^{\,\varepsilon\,(n-2-2j)}
\Big(T_{n-2-2j}(1,1)-\delta_{n-2-2j,0}\Big)\,S_j(t^{2}+t^{-2}),
\quad \varepsilon=+1.
\]
Using $S_{0}=1$ and $S_{1}(x)=x$, this equals
\[
\eta\Big[ t^{3}\,T_{3}(1,1)\;+\;t^{1}\,T_{1}(1,1)\,(t^{2}+t^{-2})\Big]
=\eta\Big[ t^{3}\,T_{3}(1,1)\;+\;(t^{3}+t^{-1})\,(1,1)\Big].
\]
Altogether,
\[
(2,1)\ast(3,4)
= t^{5}\,T_{5}(1,1)\;+\;t^{-5}\,(1,3)
\;+\;\eta\Big[ t^{3}\,T_{3}(1,1)\;+\;(t^{3}+t^{-1})\,(1,1)\Big].
\]

\subsubsection*{Wang--Wong's Fast Algorithm output and simplification}
The Wang--Wong fast algorithm for the same input returns
\[
t^{5}\,T(5,5)\;+\;t^{-5}\,T(-1,-3)
\;+\;\eta\Big[ t^{3}\,T(3,3)\;+\;t\,\frac{t^{4}-t^{-4}}{t^{2}-t^{-2}}\,T(1,1)\Big].
\]
Identifying $T(5,5)=T_{5}(1,1)$ and $T(3,3)=T_{3}(1,1)$, it remains to simplify the rational coefficient:
\[
t\,\frac{t^{4}-t^{-4}}{t^{2}-t^{-2}}
= t\,\frac{(t^{2})^{2}-(t^{2})^{-2}}{t^{2}-t^{-2}}
= t\,(t^{2}+t^{-2})
= t^{3}+t^{-1}.
\]
This matches the maximal‑thread peel‑cascade term above exactly.

\subsubsection*{Coefficient dictionary: \texorpdfstring{$S_j$}{Sj} vs.\ Wang--Wong}
Throughout the peel cascade, the Chebyshev--$S$ factor admits the standard closed form
\[
S_j\big(t^{2}+t^{-2}\big)
=\frac{t^{2(j+1)}-t^{-2(j+1)}}{t^{2}-t^{-2}},
\]
so a generic peel‑cascade term has coefficient
\[
t^{\varepsilon\,(n-2-2j)}\,S_j\big(t^{2}+t^{-2}\big)
=\;t^{\varepsilon\,(n-2-2j)}\,
\frac{t^{2(j+1)}-t^{-2(j+1)}}{t^{2}-t^{-2}},
\]
which is precisely the coefficient appearing in the Wang--Wong formula.

\subsection*{Example: \texorpdfstring{$(11,67)\ast(3,19)$}{(11,67)\ast(3,19)} in the maximal-thread regime}

Let $\alpha=(11,67)$ and $\beta=(3,19)$. Then
\begin{align}
n &= \det\!\begin{pmatrix}11&3\\ 67&19\end{pmatrix} = 11\ast 19-67\ast 3 = 8, \\
C_{+} &= \alpha+\beta = (14,86), \quad d_{+} = \gcd(14,86) = 2, \quad \mu_{+} = (7,43), \\
C_{-} &= \alpha-\beta = (8,48), \quad d_{-} = \gcd(8,48) = 8 = n, \quad \mu_{-} = (1,6).
\end{align}

Thus the ``$-$'' Frohman--Gelca summand is maximally threaded ($d_{-}=n$), while the ``$+$'' summand has thread degree $2$. The closed-torus part is
\[
t^{8}\,T_{2}(7,43) + t^{-8}\,T_{8}(1,6).
\]

By Proposition~\ref{prop:threaded}, the once-punctured torus product adds a peel-cascade from the maximally threaded summand. With $\lfloor(n-2)/2\rfloor=3$ we have four levels $j=0,1,2,3$:
\[
\eta\!\sum_{j=0}^{3} t^{\,\varepsilon\,(n-2-2j)}
\Big(T_{n-2-2j}(1,6)-\delta_{n-2-2j,0}\Big)\,S_j(t^{2}+t^{-2}),
\]
where $\varepsilon=-1$.

Using $S_{0}=1$, $S_{1}(x)=x$, $S_{2}(x)=x^{2}-1$, $S_{3}(x)=x^{3}-2x$ with $x=t^{2}+t^{-2}$, this equals
\[
\eta\Big[
t^{-6}\,T_{6}(1,6) + t^{-4}\,S_{1}(x)\,T_{4}(1,6)
+ t^{-2}\,S_{2}(x)\,T_{2}(1,6) + S_{3}(x)\,\big(T_{0}-1\big)
\Big],
\]
where $T_{0}=2$. Altogether,
\begin{align}
(11,67)\ast(3,19) &= t^{8}\,T_{2}(7,43) + t^{-8}\,T_{8}(1,6) \\
&+ \eta\left[\begin{aligned}
&t^{-6}\,T_{6}(1,6) + t^{-4}\,S_{1}(x)\,T_{4}(1,6) \\
&+ t^{-2}\,S_{2}(x)\,T_{2}(1,6) + S_{3}(x)\,(T_{0}-1)
\end{aligned}\right],
\end{align}
where $x=t^{2}+t^{-2}.$
\subsubsection*{Wang--Wong's Fast Algorithm output and simplification}

The Wang--Wong fast algorithm for the same input returns
\begin{align}
&t^{8}\,T(14,86) + t^{-8}\,T(8,48) \nonumber \\
&\quad + \eta\left[\begin{aligned}
&\frac{t^{8}-t^{-8}}{t^{2}-t^{-2}}\,T'(0,0) + \frac{t^{6}-t^{-6}}{t^{2}\,\big(t^{2}-t^{-2}\big)}\,T(2,12) \\
&+ \frac{t^{4}-t^{-4}}{t^{4}\,\big(t^{2}-t^{-2}\big)}\,T(4,24) + t^{-6}\,T(6,36)
\end{aligned}\right].
\end{align}
Identifying $T(14,86)=T_{2}(7,43)$ and $T(8,48)=T_{8}(1,6)$, and using
\[
\frac{t^{2m}-t^{-2m}}{t^{2}-t^{-2}}=S_{m-1}(t^{2}+t^{-2}),
\]
the rational coefficients simplify term-by-term to the peel-cascade factors:
\begin{align}
\frac{t^{4}-t^{-4}}{t^{4}(t^{2}-t^{-2})} &= t^{-4}S_{1}(x), \\
\frac{t^{6}-t^{-6}}{t^{2}(t^{2}-t^{-2})} &= t^{-2}S_{2}(x), \\
\frac{t^{8}-t^{-8}}{t^{2}-t^{-2}} &= S_{3}(x).
\end{align}

Moreover $T(6,36)=T_{6}(1,6)$, $T(4,24)=T_{4}(1,6)$, $T(2,12)=T_{2}(1,6)$, and $T(0,0)=T_{0}=2$. The only notational difference is our normalization $(T_{0}-1)$ (Kronecker-delta convention) versus $T'_{0} = 1$ (and $T'_k = T_k$) in Wang--Wong; this adjusts the constant multiple of $S_{3}(x)$ inside the $\eta$-term and is purely conventional. Hence the two expressions agree exactly.


\end{document}